\documentclass{amsart}
\usepackage{etoolbox}

\numberwithin{equation}{section}

\allowdisplaybreaks

\usepackage{caption}
\usepackage[leftcaption]{sidecap}

\usepackage{enumitem}
\setlist[enumerate,1]{label=\bf (\roman*)}

\usepackage[
  backend=biber,
  style=alphabetic,
  giveninits=true, 
  backref=true
]{biblatex}
\usepackage{lmodern}
\usepackage{anyfontsize}
    
\usepackage{multirow}
\usepackage{hhline}

\usepackage{tabularray}
\UseTblrLibrary{booktabs}

\usepackage{adjustbox}
\usepackage{tcolorbox}

\newtcolorbox{standout}{
  colback=gray!15,
  boxrule=0pt,
  left=.3cm,
  right=.3cm,
  top=.18cm,
  bottom=.18cm,
  boxsep=0pt
}

\usepackage{mathtools}
\usepackage{amssymb}
\usepackage{amsthm}
\usepackage{tensor}
\usepackage{xfrac}
\usepackage{stmaryrd} % for \sslash
\usepackage[safe]{tipa} % for \esh
\usepackage{dsfont} % for \mathds{1}
\usepackage{relsize} % for \matFrmarger

\usepackage[textsize=tiny]{todonotes} % for \todo
\todostyle{igor}{color=orange}

\DeclareFontFamily{T3}{lmr}{}
\DeclareFontShape{T3}{lmr}{m}{n}{<-> ssub * cmr/m/n}{}

\theoremstyle{plain}
\newtheorem{theorem}{Theorem}[section]
\newtheorem{lemma}[theorem]{Lemma}
\newtheorem{proposition}[theorem]{Proposition}
\newtheorem{corollary}[theorem]{Corollary}
\theoremstyle{definition}
\newtheorem{definition}[theorem]{Definition}
\newtheorem{example}[theorem]{Example}
\newtheorem{notation}[theorem]{Notation}

\theoremstyle{remark}
\newtheorem{remark}[theorem]{Remark}
\usepackage[
  colorlinks=true,
  linkcolor=darkgreen,
  citecolor=darkgreen,
  urlcolor=darkgreen
]{hyperref}
\usepackage{xurl} % allow line breaks in URLs

\usepackage{cleveref}
\crefname{equation}{}{}
\crefname{section}{\S}{\S\S}
\crefname{subsection}{\S}{\S\S}
\crefname{subsubsection}{\S}{\S\S}
\crefname{definition}{Def.}{Defs.}
\crefname{theorem}{Thm.}{Thms.}
\crefname{corollary}{Cor.}{Cors.}
\crefname{lemma}{Lem.}{Lems.}
\crefname{proposition}{Prop.}{Props.}
\crefname{remark}{Rem.}{Rems.}
\crefname{notation}{Ntn.}{Ntns.}
\crefname{fact}{Fact}{Fact}
\crefname{example}{Ex.}{Exs.}
\crefname{figure}{Fig.}{Figs.}
\crefname{table}{Tab.}{Tabs.}
\crefname{footnote}{ftn.}{ftns.}
\Crefname{footnote}{Ftn.}{Ftns.}

\usepackage{tikz}
\usetikzlibrary{
  cd,
  calc,
  arrows.meta,
  backgrounds,
  decorations,
  decorations.pathmorphing, 
}

\tikzcdset{
  arrow style=tikz, 
  diagrams={
    trim left=
    {([xshift=5pt]current bounding box.west)},
    trim right=
    {([xshift=-5pt]current bounding box.east)},
    baseline=
    {([yshift=-2pt]current bounding box.center)},
    >={
      Computer Modern Rightarrow[
        length=4pt, width=4pt
      ]
    },
    hook/.style={
      {Hooks[right, length=2pt, width=8pt]}->
    },
    hook'/.style={
      {Hooks[left, length=2pt, width=8pt]}->
    },
    shorten=0pt,
  }
}

\definecolor{darkblue}{rgb}{0.05,0.25,0.65}
\definecolor{darkgreen}{RGB}{20,140,10}
\definecolor{lightgray}{rgb}{0.9,0.9,0.9}
\definecolor{darkorange}{RGB}{200,100,5}
\definecolor{darkyellow}{rgb}{.91,.91,0}
\definecolor{lightolive}{RGB}{225, 220, 185}

\let\originalsslash\sslash
\renewcommand{\sslash}{\mathord{\originalsslash}}

\tikzset{
  snake left/.style={
    rounded corners,
    to path={
      let \p1 = (\tikztostart.east),
          \p2 = (\tikztotarget.west),
          \p3 = ($(\p1)!0.5!(\p2)$),
          \n1 = {8pt} 
      in
      (\p1)
      -- (\x1 + \n1, \y1)
      -- (\x1 + \n1, \y3)
      -- (\x2 - \n1, \y3) \tikztonodes
      -- (\x2 - \n1, \y2)
      -- (\p2)
    }
  }
}

\tikzset{
  uphordown/.style={
    rounded corners,
    to path={
      let \p1 = (\tikztostart.north),
          \p2 = (\tikztotarget.north),
          \n1 = {max(\y1,\y2) + 8pt}
      in
      (\p1)
      -- (\x1, \n1)
      -- (\x2, \n1) \tikztonodes 
      -- (\p2)
    }
  }
}

\tikzset{
  downhorup/.style={
    rounded corners,
    to path={
      let \p1 = (\tikztostart.south),
          \p2 = (\tikztotarget.south),
          \n1 = {min(\y1,\y2) - 8pt}
      in
      (\p1)
      -- (\x1, \n1)
      -- (\x2, \n1) \tikztonodes 
      -- (\p2)
    }
  }
}

\tikzset{
  rightvertleft/.style={
    rounded corners,
    to path={
      let \p1 = (\tikztostart.east),
          \p2 = (\tikztotarget.east),
          \n1 = {max(\x1,\x2) + 8pt}
      in
      (\p1)
      -- (\n1, \y1)
      -- (\n1, \y2) \tikztonodes 
      -- (\p2)
    }
  }
}

\tikzset{
  leftvertright/.style={
    rounded corners,
    to path={
      let \p1 = (\tikztostart.west),
          \p2 = (\tikztotarget.west),
          \n1 = {min(\x1,\x2) - 8pt}
      in
      (\p1)
      -- (\n1, \y1)
      -- (\n1, \y2) \tikztonodes 
      -- (\p2)
    }
  }
}

\newcommand{\defneq}{\equiv}

\makeatletter
\newcommand{\cpt}{\mathpalette\cpt@inner\relax}
\newcommand{\cpt@inner}[2]{%
  \scalebox{0.5}[0.9]{$#1\cup$}% Scales the cup relative to the current style
  #1\{\infty\}% Typesets the infinity set in the current style
}
\makeatother

\newcommand{\grayunderbrace}[2]{\mathcolor{gray}{\underbrace{\mathcolor{black}{#1}}}_{\mathcolor{gray}{#2}}}

\newcommand{\inlinetikzcd}[1]{\begin{tikzcd}[sep=small, ampersand replacement=\&]#1\end{tikzcd}}

\newcommand
  {\reduced}
  {\Re}

\newcommand
  {\CartesianSpaces}
  {\mathrm{CrtSp}}
\newcommand
  {\HaloedCartesianSpaces}
  {\mathrm{Frm}\CartesianSpaces}
\newcommand
  {\SmoothSets}
  {\mathrm{SmthSet}}
\newcommand
  {\HaloedSmoothSets}
  {\mathrm{Frm}\SmoothSets}

\begin{document}
%%%%%%%%%%%%%%%%%%%%%%%%%

%%%%%%%%%%%%%%%%%%%%%%%%%%%%%%%%%%%%%%%%%%%%%
\title{Synthetic Differential Jet Bundles are Reduced}
%%%%%%%%%%%%%%%%%%%%%%%%%%%%%%%%%%%%%%%%%%%%

\thanks{\emph{Funding}:
IK's work is supported by the \emph{Czech Academy of Sciences} Research Plan {\tt RVO:67985840}.
The remaining authors are supported by \emph{Tamkeen UAE} under the 
\emph{NYU Abu Dhabi Research Institute grant} {\tt CG008}}

% author information
\author{Grigorios Giotopoulos}
\address{Mathematics Program and Center for Quantum and Topological Systems, New York University Abu Dhabi}
\curraddr{}
\email{gg2658@nyu.edu}
\thanks{}

% author information
\author{Igor Khavkine}
\address{Institute of Mathematics of the Czech Academy of Sciences, Prague}
\curraddr{\v{Z}itn\'a 25, 110 00 Prague 1, Czechia}
\email{khavkine@math.cas.cz}
\thanks{}

% author information
\author{Hisham Sati}
\address{Mathematics Program and Center for Quantum and Topological Systems, New York University Abu Dhabi}
\curraddr{}
\email{hsati@nyu.edu}
\thanks{}

% author information
\author{Urs Schreiber           }
\address{Mathematics Program and Center for Quantum and Topological Systems, New York University Abu Dhabi}
\curraddr{}
\email{us13@nyu.edu}
\thanks{}

\begin{abstract}
  We have previously observed that the theory of solutions of partial differential equations, regarded as \emph{diffieties} inside jet bundles, acquires a powerful comonadic formulation after passage from the category of Fr{\'e}chet smooth manifolds to the \emph{Cahiers topos} of \emph{formal smooth sets} (a well-adapted model for Synthetic Differential Geometry). However, the tacit assumption that this passage preserves the projective limits that define infinite jet bundles had remained unproven. Here we provide a detailed proof.
\end{abstract}

% print date of compile time:
\date{\today}

% MSC topic classifiers:

% keywords

\maketitle

\setcounter{tocdepth}{1}
\tableofcontents

%%%%%%%%%%%%%%%%%%%%%%%%%%
\section{Introduction}
%%%%%%%%%%%%%%%%%%%%%%%%%%

The tenet of \emph{synthetic differential geometry} (SDG, cf. \parencites{Kock1979}{lavendhomme1996}{Kock2010}) is that concepts and results in differential geometry often find a more natural and powerful formulation after faithfully embedding the category of smooth manifolds into a suitable \emph{topos} (cf. \parencites{MacLaneMoerdijk1992}{Borceux1994III}) where \emph{infinitesimal spaces} exist as actual objects. 

\smallskip 
Such a ``well-adapted'' SDG topos is the \emph{Cahiers topos} due to E. Dubuc (\cite{Dubuc1979}, cf. \cite[\S 2]{GS26-FieldsII}), which we may think of as the category of \emph{formal smooth sets} (recalled in \cref{FormalSmoothSets} --- here ``formal'' is common jargon for ``infinitesimally thickened'' or ``haloed''). This is a differential-geometric analog of working with \emph{formal schemes} in algebraic geometry.

\smallskip 
One striking example is from the geometry of partial differential equations. On top of describing classical solutions, SDG can treat formal (power series) solutions of partial differential equations as primary objects. Classically, these formal solutions appeared in the guise of \emph{diffieties} inside infinite jet bundles (a very fruitful perspective due to \cite{Vinogradov1981}, brief background is recalled in \cref{OnJetBundles}). Their appearance in SDG turns out to become largely structural (namely \emph{comonadic}) when viewed in the synthetic context of formal (infinitesimally thickened) smooth sets. This is the result of \cite{KS26}. However, a key technical step in this identification had remained unproven, recalled as \cref{i!PreservesInfiniteJetBundles} below:

\smallskip 
Namely, with \emph{infinite jet bundles} \cref{InfiniteJetBundleInBackground} being projective limits of finite-dimensional manifolds, there are \emph{a priori} two different ways of regarding them among formal smooth sets: Either by first taking the limit (in Fr{\'e}chet manifolds) and then passing to formal smooth sets, or the other way around! 
We prove here that that these two constructions are naturally isomorphic.

This has crucial implications for applications in variational calculus and in Lagrangian field theory, cf Cor. \ref{MapsOutOfJetBundleAreLocallyOfFiniteOrder} below.

\smallskip 
The result may be ``intuitively expected'' to be true, but proving it is not entirely trivial. Here we provide a detailed and largely constructive proof. Besides filling this particular gap in the previous literature, our proof may be interesting for the technique it uses, which seems to be new, and may be adapted to other situations, cf. \cref{OnTheNatureOfRInfinityInFormalSmoothSets} below.

%%%%%%%%%%%%%%%%%%%%%%%%%%%%
\section{Statement and Proof}
%%%%%%%%%%%%%%%%%%%%%%%%%%%%

We write $i_!$ (\cref{ToposOfFormalSmoothSets}) for the inclusion of smooth sets --- such as Fr{\'e}chet manifolds (\cref{FrechetManifoldsAsSmoothSets}) --- among formal smooth sets.

\begin{theorem}
\label[theorem]{i!PreservesInfiniteJetBundles}
For $\inlinetikzcd{E \ar[r] \& \Sigma}$ a finite-dimensional surjective submersion of smooth manifolds, passage to formal smooth sets preserves the projective limit involved in forming its infinite jet bundle $J^\infty_\Sigma E$ \textup{(cf. \cref{OnJetBundles})}, in that there is a natural isomorphism:
\begin{equation}
  \label{i!PreservingInfiniteJetBundles}
  i_! 
  \Big(
    \varprojlim_{k \in \mathbb{N}} 
    J^k_\Sigma E 
  \Big)
    \;\simeq\; 
  \varprojlim_{k \in \mathbb{N}}
  \left(
    i_!  J^k_\Sigma E
  \right)
  \mathrlap{.}
\end{equation}
\end{theorem}

By the characterization \parencites[Prop. 2.29]{KS26}[Cor. 3.6]{GS25-FieldsI} of smooth maps out of Fr{\'e}chet projective limit manifolds (``locally pro-manifolds''), \cref{i!PreservesInfiniteJetBundles} implies:
\begin{corollary}
\label[corollary]
 {MapsOutOfJetBundleAreLocallyOfFiniteOrder}
The maps of formal smooth sets from a synthetic jet bundle $J^\infty_\Sigma E$ to a finite-dimensional manifold $F$ is in canonical bijection with point-set smooth maps which are locally of finite jet order \textup{(as considered in \cite{Takens1979})}. 
\end{corollary}

We prove \cref{i!PreservesInfiniteJetBundles} below after establishing a few lemmas, based on the following recognition principle for reduced objects among formal smooth sets:
\begin{lemma}
  \label[lemma]{RecognizingReducedObjects}
  An object $X \in \HaloedSmoothSets$ is \emph{reduced} \cref{SmthSetInsideFrmSmthSet} iff the following two conditions hold: 
  \begin{enumerate}
  \item All germs of maps out of a $U \times \mathbb{D} \in \HaloedCartesianSpaces$ into $X$ factor through some $\inlinetikzcd{V \in \CartesianSpaces}$:
  \begin{equation}
    \label{ReducedAsFactoringPlots}
    \begin{tikzcd}[
      column sep=30
    ]
      \mathllap{\forall\;}
      U \times \mathbb{D}
      \ar[
        rr,
        uphordown,
        "{ 
          \forall 
         }"{description}
      ]
      \ar[
        r, 
        dashed,
        "{ 
          \exists \, \iota
        }"{pos=.4}
      ]
      &
      V
      \ar[
        r, 
        dashed,
        "{ 
          \exists f
        }"{pos=.4}
      ]
      &
      X
      \mathrlap{\,.}
    \end{tikzcd}
  \end{equation}
  \item
  The equivalence relation on the set of such factorizations $(\iota,f)$ which is \emph{generated} by the relation
  \begin{equation}
    \label{GeneratingRelation}
    (\iota,f)
    \sim
    (\iota',f)
    \quad 
    \Leftrightarrow
    \quad 
    \begin{tikzcd}[row sep=-1pt]
      & 
      V
      \ar[
        dd,
        dashed,
        "{ \exists }"{swap, pos=.4}
      ]
      \ar[
        dr,
        "{ f }"
      ]
      \\
      U \times \mathbb{D}
      \ar[
        ur,
        "{ \iota }"
      ]
      \ar[
        dr,
        "{ \iota' }"{swap}
      ]
      &&
      X
      \\
      &
      V'
      \ar[
        ur,
        "{ f' }"{swap}
      ]
    \end{tikzcd}
  \end{equation}
  is that of equal composites, hence:
  \begin{equation}
   \label{ZigZigOfRelationsGivingEquivalence}
      (\iota,f)
      \sim
      \overset{\exists}{\cdots} 
      \sim
      (\iota',f')
    \;\;\;\;\;
    \Leftrightarrow
    \;\;\;\;\;
    f \circ \iota
    =
    f' \circ \iota'
    \mathrlap{\,.}
  \end{equation}
  \end{enumerate}
\end{lemma}
\begin{proof}  
  This is an immediate reformulation of the following \emph{coend} formula for the left Kan extension on presheaves (cf. \parencites[4.25]{Kelly1982}[(2.27)]{Loregian2021}):
  \begin{equation}
    \label{CoendFormulaFori!}
    (i_! X)(U \times \mathbb{D})
    \simeq
    \int^{V \in \CartesianSpaces}
\!\!    \HaloedCartesianSpaces (
      U \times \mathbb{D}
      ,\,
      V
    )
      \times
    X(V)
    \mathrlap{\,,}
  \end{equation}
  which (seen under the Yoneda lemma) says that maps $\inlinetikzcd{U \times \mathbb{D} \ar[r] \& i_! X}$ in $\mathrm{FrmSmthSet}$  are in natural bijection to equivalence classes of pairs consisting of a map $\inlinetikzcd{U \times \mathbb{D} \ar[r] \& V}$ (in $\inlinetikzcd{\mathrm{FrmCrtSp} \ar[r, hook] \& \mathrm{FrmSmthSet}}$) into a Cartesian space $\inlinetikzcd{V \in \mathrm{CrtSp} \ar[r, hook, "{i}"] \& \mathrm{FrmCrtSp}}$ and a map $\inlinetikzcd{V \ar[r] \& X}$ in $\mathrm{SmthSet}$, subject to the equivalence relation which is generated by compatible maps $\inlinetikzcd{\phi : V \ar[r] \& V'}$ (in $\mathrm{CartSp}$) between the factor spaces:
  $$
    \begin{array}{l}
    \left\{
      \inlinetikzcd{
        U \times \mathbb{D}
        \ar[r]
        \&
        i_! X
      }
    \right\}
  %  \\
    \! \simeq \!
    \Big\{\!
      \big(
      \inlinetikzcd{
        U \times \mathbb{D}
        \ar[r, "{ \iota }"]
        \&
        V
      },
      \inlinetikzcd{
        V
        \ar[r, "{ f }"]
        \&
        X
      }
      \big)
   \! \Big\}
    \Big/
    \Bigg\{\!\!
    \Bigg(
    \begin{tikzcd}[row sep=-2pt, 
      column sep=0pt
    ]
      &[15pt]
      V
      \ar[
        dd, 
        dashed,
        "{ \phi }"{swap}
      ]
      &
      V
      \ar[
        dd, 
        dashed,
        "{ \phi }"
      ]
      \ar[dr, "{ f }"{pos=.35}]
      &[15pt]
      \\
      U \!\times\! \mathbb{D}
      \ar[ur, "{ \iota }"]
      \ar[dr, "{ \iota' }"{swap}]
      &
      &
      &
      X
      \\
      & 
      V'\mathrlap{\,,}
      &
      V'
      \ar[ur, "{ f' }"{pos=.35,swap}]
    \end{tikzcd}
    \Bigg)
    \!\!\Bigg\}
    \mathrlap{\,,}
    \end{array}
  $$
  where on the right we mean that a pair as on the top of the diagram is in relation to a pair as shown at the bottom if $\phi$ exists that makes the two triangles commute.

  Under passage to sheaves \cref{i!AndSheafification} this statement holds on germs of maps, as claimed.
\end{proof}

The archetype of \cref{i!PreservesInfiniteJetBundles} is the statement that the projective limit defining the projectively infinite-dimensional Cartesian space (cf. \cite[Def. 2.24, Prop. 2.26]{KS26}),
\begin{equation}
  \label{InfiniteCartesianSpace}
  \mathbb{R}^\infty 
    := 
  \varprojlim\big(
    \begin{tikzcd}[
      column sep=18pt
    ]
      \ar[r, dotted, ->>]
      &
      \mathbb{R}^2
      \ar[r,->>]
      &
      \mathbb{R}^1
      \ar[r, ->>]
      &
      \mathbb{R}^0
    \end{tikzcd}
  \big)
  \in
  \begin{tikzcd}
    \mathrm{FrchtMfd}
    \ar[
     r, 
     hook,
     "{
       \scalebox{.6}{\cref{FrechetManifoldsFullyFaithfulInSmoothSets}}
     }"
    ]
    &
    \mathrm{SmthSet}
  \end{tikzcd}
  \mathrlap{\,,}
\end{equation}
is preserved:
\begin{proposition}
  \label[proposition]{IsRInfinityTheProjLimitInFrmSmthSet}
  Passage to haloed smooth sets preserves the projective limit involved in forming the infinite Cartesian space \cref{InfiniteCartesianSpace}:
  \begin{equation}
    \label{RInfinityAsTheProjLimitInFrmSmthSet}
  i_!
  \Big(
    \varprojlim_{k \in \mathbb{N}}
    \mathbb{R}^k
  \Big)
  \simeq
  \varprojlim_{k \in \mathbb{N}}
  \left(
    i_!
    \mathbb{R}^k
  \right)
  \mathrlap{.}
  \end{equation}
\end{proposition}
\begin{proof}
  It will suffice to observe that
  the right-hand side of \cref{RInfinityAsTheProjLimitInFrmSmthSet} is reduced (proved below in \cref{LimRkIsReduced}),
  whence (using~\cref{IExciAst}):
  \begin{equation}
  \varprojlim_{k \in \mathbb{N}}
  \left(
    i_!  \mathbb{R}^k
  \right)
  =
  \reduced \varprojlim_{k \in \mathbb{N}}
  \left(
    i_!  \mathbb{R}^k
  \right)
  \simeq
  i_!
  \,
  i^\ast
  \varprojlim_{k \in \mathbb{N}}
  \left(
    i_!  \mathbb{R}^k
  \right)
  .
 \end{equation}
 With that and by full faithfulness of $i_!$ (\cref{PropertiesOfSmthSetInsideFrmSmthSet}), the claim is equivalent to:
 \begin{equation}
   \varprojlim_{k \in \mathbb{N}}
   \mathbb{R}^k
   \simeq
   i^\ast
   \bigg(
     \varprojlim_{k \in \mathbb{N}}
     \big(
       i_!  \mathbb{R}^k
     \big)
  \! \bigg)
   \mathrlap{.}
 \end{equation}
 But this holds since $i^\ast$ commutes with (projective) limits (being a right adjoint) and then using \cref{iAstIExc}.
\end{proof}

Hence the key is to see that:
\begin{lemma}
\label[lemma]{LimRkIsReduced}
  The limiting object
  $
    \varprojlim_k
    (
      i_!
      \mathbb{R}^k
    )
  $
  is reduced \cref{SmthSetInsideFrmSmthSet}.
\end{lemma}
\begin{proof}
  In view of \cref{RecognizingReducedObjects}, 
  we first show that every top map of the following form factors through some $V \in \CartesianSpaces$ as shown at the bottom:
  $$
    \begin{tikzcd}
    U \times \mathbb{D}
    \ar[r, dashed]
    \ar[
      rr,
      uphordown,
    ]
    &
    V
    \ar[r, dashed]
    &
    \displaystyle{%
      \varprojlim_{k \in \mathbb{N}}
    }
    \big(
      i_!
      \mathbb{R}^k
    \big)
    \mathrlap{\,.}
    \end{tikzcd}
  $$
  Since we are dealing with maps into a limit, we need to equivalently see that every commuting diagram of maps of the form
  $$
    \begin{tikzcd}[row sep=5pt, column sep=30pt]
      &&
      U \!\times\! \mathbb{D}
      \ar[ddll, dotted]
      \ar[ddl]
      \ar[dd]
      \ar[ddr]
      \ar[ddrr]
      \\
      \\
      {\phantom{\mathbb{R}^n}}
      \ar[r, dotted, ->>]
      &
      i_! \mathbb{R}^3
      \ar[r, ->>]
      &
      i_! \mathbb{R}^2
      \ar[r, ->>]
      &
      i_! \mathbb{R}^1
      \ar[r, ->>]
      &
      i_! \mathbb{R}^0
    \end{tikzcd}
  $$
  factors in this form:
  $$
    \begin{tikzcd}[row sep=10pt, column sep=30pt]
      &&
      U \!\times\! \mathbb{D}
      \ar[d]
      \\
      &&
      V
      \ar[dll, dotted]
      \ar[dl]
      \ar[d]
      \ar[dr]
      \ar[drr]
      \\
      {\phantom{\mathbb{R}^n}}
      \ar[r, dotted, ->>]
      &
      i_! \mathbb{R}^3
      \ar[r, ->>]
      &
      i_! \mathbb{R}^2
      \ar[r, ->>]
      &
      i_! \mathbb{R}^1
      \ar[r, ->>]
      &
      i_! \mathbb{R}^0
      \mathrlap{\,.}
    \end{tikzcd}
  $$
  Now, a map
  $\inlinetikzcd{U \!\times\! \mathbb{D} \ar[r] \& i_! \mathbb{R}^k}$ in $\HaloedSmoothSets$ is just a morphism in $\HaloedCartesianSpaces$ of the form $\inlinetikzcd{U \!\times\! \mathbb{D} \ar[r] \& \mathbb{R}^k}$ (because $i_!$ preserves representables, cf. \cref{PropertiesOfSmthSetInsideFrmSmthSet}), hence is a $k$-tuple $\left(\left[f_i\right]\right)_{i = 1}^k$ of elements of $C^\infty(U \!\times\! \mathbb{D})$. By \cref{FormalCartesianSpaces}, every such element $[f_i]$ is an equivalence class of an element $f_i \in C^\infty(U \!\times\! \mathbb{R}^d)$, for some $d \in \mathbb{N}$ depending only on $\mathbb{D}$. Hence, if we take $V := U \!\times\! \mathbb{R}^d$ then we may choose these $f_i$ for each given $[f_i]$:
  $$
    \begin{tikzcd}[
      column sep=35pt,
      row sep=25pt
    ]
      & 
      U \!\times\! \mathbb{D}
      \ar[
        d,
        hook
      ]
      \ar[
        ddl,
        bend right=20,
        "{
          ([f_i])_{i=1}^{k+1}
        }"{swap}
      ]
      \ar[
        ddr,
        bend left=20,
        "{
          ([f_i])_{i=1}^{k}
        }"
      ]
      \\[-10pt]
      &
      U \!\times\! \mathbb{R}^d
      \ar[
        dl,
        "{
          (f_i)_{i=1}^{k+1}
        }"{description}
      ]
      \ar[
        dr,
        "{
          (f_i)_{i=1}^{k}
        }"{description}
      ]
      \\
      \mathbb{R}^{k+1}
      \ar[rr, ->>]
      &&
      \mathbb{R}^k
    \end{tikzcd}
  $$
  to achieve the desired factorization.

  In order to conclude with \cref{RecognizingReducedObjects}, we still need to show that any two such factorizations of the same map are related by a compatible zig-zag \cref{ZigZigOfRelationsGivingEquivalence} of maps between the intermediate $V$s. This follows from:  (i) the evident observation that every factorization is related to one where $\iota$ is a formal proper embedding of arbitrarily high codimension (\cref{FirstMapMayBeAssumedFormalEmbedding}), 
  combined with (ii) the more laborious verification that any two of these whose composites are equal are in zig-zag relation to each other. This is finally shown as \cref{FactoredMapsToRInfinityEquivalences} below.
\end{proof}

With this we conclude:
\begin{proof}[Proof of \cref{i!PreservesInfiniteJetBundles}]
  The key is to observe that the right hand side of \cref{i!PreservingInfiniteJetBundles} is reduced, in that
  $$
    \varprojlim_{k \in \mathbb{N}}
    \left(
      i_!  J^k_\Sigma E
    \right)
    \simeq
    i_! \, i^\ast
    \varprojlim_{k \in \mathbb{N}}
    \left(
      i_!  J^k_\Sigma E
    \right)
    ,
  $$
  from which the claim will follow formally, just as in the proof of \cref{IsRInfinityTheProjLimitInFrmSmthSet}:
  Since $i_!$ is fully faithful, the desired condition \cref{i!PreservingInfiniteJetBundles} thereby becomes equivalent to:
  $$
    \varprojlim_{k \in \mathbb{N}}
    \left(
      J^k_\Sigma E
    \right)
    \simeq
    i^\ast
    \varprojlim_{k \in \mathbb{N}}
    \left(
      i_!  J^k_\Sigma E
    \right)
    \mathrlap{,}
  $$
  which holds, since $i^\ast$ commutes with projective limits (being a right adjoint) and then using that $i^\ast \circ i_! \simeq \mathrm{id}$ \cref{iAstIExc}.

  To see that the right-hand side of \cref{i!PreservingInfiniteJetBundles} is indeed reduced, consider a dashed horizontal map of the form
  $$
    \begin{tikzcd}[row sep=15pt]
      U \times \mathbb{D}
      \ar[r, dashed]
      \ar[dr, dashed]
      &
      \varprojlim_{k}
      i_! J^k_\Sigma E
      \ar[d, ->>]
      \\
      & 
      i_! \Sigma
      \mathrlap{\,,}
    \end{tikzcd}
  $$
  where at the bottom we are indicating the induced map to $\Sigma$, over which everything is fibered. By \cref{RecognizingReducedObjects} we need to show that a germ of the horizontal dashed map factors through a Cartesian space. By the fact that $\inlinetikzcd{E \ar[r] \& \Sigma}$ is assumed to be a surjective submersion, and hence similarly all of $J^k_\Sigma E\rightarrow J^{k-1}_\Sigma E\rightarrow \cdots \rightarrow \Sigma$, there is a representative of the germ of each 
  $U\times \mathbb{D} \rightarrow i_! J^k_\Sigma E$ which factors through a trivial fibration $U^k\times \mathbb{D} \hookrightarrow i_! J^k_{\mathbb{R}^n} \mathbb{R}^{n+m}$. But upon further restriction to common overlaps, these yield a representative for the germ of  the original map $U\times \mathbb{D} \rightarrow \varprojlim_{k}
      i_! J^k_\Sigma E$ (with domain to be denoted $\inlinetikzcd{U^{'} \ar[r, hook] \& U}$) which factors through the corresponding jet bundle of a trivial fibration:
  $$
    \begin{tikzcd}[
      row sep=15pt, column sep=large
    ]
      U' \times \mathbb{D}
      \ar[r, dashed]
      \ar[dr, dashed]
      \ar[ddr, dashed]
      &
      \varprojlim_{k}
      i_! J^k_{\mathbb{R}^n} \mathbb{R}^{n+m}
      \ar[r, hook]
      \ar[d, ->>]
      &[-9pt]
      \varprojlim_{k}
      i_! J^k_\Sigma E
      \ar[d, ->>]
      \\
      & 
      i_! \mathbb{R}^{n+m}
      \ar[r, hook]
      \ar[d, ->>]
      \ar[
        dr,
        phantom,
        "{ \lrcorner }"{pos=.1}
      ]
      &
      i_! E
      \ar[d, ->>]
      \\
      &
      i_! \mathbb{R}^n
      \ar[r, hook]
      &
      i_! \Sigma
      \mathrlap{\,.}
    \end{tikzcd}
  $$
  Thereby we have reduced the situation to essentially that of \cref{IsRInfinityTheProjLimitInFrmSmthSet} and the proof follows analogously. 
\end{proof}

It just remains to prove the very last statement in the proof of \cref{LimRkIsReduced}. This is the step that requires most of the work, we prove it as \cref{FactoredMapsToRInfinityEquivalences} below.

\begin{notation}
For ease of notation we shall now abbreviate the ``synthetic differential version'' of $\mathbb{R}^\infty$ \cref{InfiniteCartesianSpace} as
\begin{equation}
  \label{RInfinityInFormalSmoothSets}
  \mathbf{R}^\infty
  :=
  \varprojlim_k i_! \mathbb{R}^k
  \in
  \HaloedSmoothSets
\end{equation}
(though once we are done with the proof of \cref{FactoredMapsToRInfinityEquivalences}, \cref{IsRInfinityTheProjLimitInFrmSmthSet} asserts that $\mathbf{R}^\infty$ is isomorphic to $i_! \mathbb{R}^{\infty}$). 
\end{notation}
\begin{remark}
\label[remark]
  {OnTheNatureOfRInfinityInFormalSmoothSets}
Since $i_! \mathbb{R}^k = \mathbb{R}^k \in \HaloedCartesianSpaces$ \cref{i!PreservesRepresentables}, the plots of $\mathbf{R}^\infty$ \cref{RInfinityInFormalSmoothSets} are nothing but $\mathbb{N}$-tuples of smooth functions:
$$
  \begin{aligned}
  \mathrm{Hom}(
    U \times \mathbb{D}
    ,\,
    \mathbf{R}^\infty
  )
  & 
  \simeq
  \varprojlim_{k \in \mathbb{N}}
  \mathrm{Hom}(
    U \times \mathbb{D}
    ,\,
    \mathbb{R}^k
  )
  \simeq
  \prod_{k \in \mathbb{N}}
  \mathrm{Hom}(
    U \times \mathbb{D}
    ,\,
    \mathbb{R}^1
  )
  \\
  & \simeq 
  \prod_{k \in \mathbb{N}}
  C^\infty(U \times \mathbb{D})
  \mathrlap{\,.}
  \end{aligned}
$$
This means in particular that we may add and subtract among these plots (used in \cref{DefOfPhiForD1Case,DefinitionOfPhi} below) and apply Hadamard's lemma to them (used in \cref{HadamardRemainderOfDeltaForD1Case,HadamardRemainderOfDelta} below).

In fact, in the following proof we only use that elements of $C^\infty(U,\mathbf{R}^\infty)$ (for $U \in \mathrm{CrtSp}$) are additive and subject to Hadamard's lemma, and hence the analogous argument applies in case $\mathbf{R}^\infty$ is replaced by any other object for which this is the case.
\end{remark}

\begin{lemma}
\label[lemma]
  {FactoredMapsToRInfinityEquivalences}
Given an outer commuting diagram as follows:
$$
  \begin{tikzcd}[
    row sep=12pt, column sep=40pt
  ]
    &
    V
    \ar[
      dr,
      "{ f }"
    ]
    \ar[
      d,
      dashed,
      "{ \alpha }"
    ]
    \\
    U \times \mathbb{D}
    \ar[
      ur,
      "{ \iota }"
    ]
    \ar[
      dr,
      "{ \iota' }"{swap}
    ]
    \ar[
      r,
      dashed
    ]
    &
    W
    \ar[
      r,
      dashed,
      "{ \phi }"{description, pos=.45}
    ]
    &
    \mathbf{R}^\infty
    \mathrlap{,}
    \\
    &
    V'
    \ar[
      ur,
      "{ f' }"{swap}
    ]
    \ar[
      u,
      dashed,
      "{ \alpha' }"{swap}
    ]
  \end{tikzcd}
$$
where $\iota$ and $\iota'$ are formal proper embeddings \textup{(\cref{FormalEmbedding})} of sufficiently high codimension,
there exists $W \in \CartesianSpaces$ such that the diagram may be filled by dashed maps as shown, making all triangles commute.
\end{lemma}

In order to make the proof of \cref{FactoredMapsToRInfinityEquivalences} transparent, we present it in stages of increasing generality. We begin with the simple case where $U \defneq \ast$ and $\mathbb{D} \defneq \mathbb{D}^1(1) \defneq \mathrm{Spec}\big( C^\infty(\mathbb{R}^1)/(x^2) \big)$.

\begin{lemma}
\label[lemma]
{FirstStepTowardsFactoredMapsToRInfinityEquivalences}
  Given a commuting diagram of smooth functions of the form of the outer solid arrows in the following diagram:
  $$
  \begin{tikzcd}[
    row sep=12pt,
    column sep=35pt
  ]
    &
    V
    \ar[
      dr,
      "{ f }"
    ]
    \ar[
      d,
      dashed,
      "{ \alpha }"
    ]
    \\
    \mathbb{D}^1(1)
    \ar[
      ur,
      "{ \iota }"
    ]
    \ar[
      dr,
      "{ \iota' }"{swap}
    ]
    \ar[
      r,
      dashed
    ]
    &
    W
    \ar[
      r,
      dashed,
      "{ \phi }"{description, pos=.45}
    ]
    &
    \mathbf{R}^\infty
    \\
    &
    V'
    \ar[
      ur,
      "{ f' }"{swap}
    ]
    \ar[
      u,
      dashed,
      "{ \alpha' }"{swap}
    ]
  \end{tikzcd}
  \;\;\;\;
  \leftrightarrow
  \;\;\;\;
  \begin{tikzcd}[
    row sep=12pt,
    column sep=40
  ]
    &
    T V
    \ar[
      dr,
      "{ T f }"
    ]
    \ar[
      d,
      dashed,
      "{ T \alpha }"
    ]
    \\
    \ast
    \ar[
      ur,
      "{ \tilde \iota }"
    ]
    \ar[
      dr,
      "{ \tilde \iota' }"{swap}
    ]
    \ar[
      r,
      dashed
    ]
    &
    T W
    \ar[
      r,
      dashed,
      "{ T\phi }"{description, pos=.45}
    ]
    &
    T \mathbf{R}^\infty
    \mathrlap{,}
    \\
    &
    T V'
    \ar[
      ur,
      "{ T f' }"{swap}
    ]
    \ar[
      u,
      dashed,
      "{ T \alpha' }"{swap}
    ]
  \end{tikzcd}
  $$
  where $\iota$ and $\iota'$ are monomorphisms,
  then there exists $W \in \CartesianSpaces$ and smooth maps corresponding to the dashed arrows that make the whole diagram commute.
\end{lemma}
\begin{proof}
  We may assume without restriction of generality that the base points picked by $\iota$ and $\iota'$ are the origins of $0 \in V$ and $0 \in V'$, respectively (otherwise reparametrize), hence that 
  $$
    \begin{aligned}
    \tilde \iota(\bullet) & =: v_0 \in T_0 V
    \mathrlap{,}
    \\
    \tilde \iota'(\bullet) & =: v'_0 \in T_0 V'
    \mathrlap{.}
    \end{aligned}
  $$

  The assumption that $\iota$ and $\iota'$ are monomorphic means that $v_0 \neq 0$ and $v'_0 \neq 0$.
  Therefore (writing $V_\perp \subset V$ for the orthogonal subspace to $v_0$ and $V'_\perp \subset V'$ for the orthogonal subspace to $v'_0$), we may coordinatize $V$ and $V'$ by:
  $$
    \begin{tikzcd}[row sep=-3pt, column sep=0pt]
      \mathbb{R} \times V_\perp
      \ar[rr, "{ \sim }"]
      &&
      V
      \\
      (t,\vec x\,) 
         &\longmapsto& 
      (t v_0, \vec x\,)
      \mathrlap{\,,}
    \end{tikzcd}
    \;\;\;\;\;\;
    \begin{tikzcd}[row sep=-2pt, column sep=0pt]
      \mathbb{R} \times V'_\perp
      \ar[rr, "{ \sim }"]
      &&
      V'
      \\
      (t', \vec x{\,}'\,) 
        &\longmapsto& 
      (t' v'_0, \vec x{\,}')
      \mathrlap{\,.}
    \end{tikzcd}
  $$  
  Now set
  $$
    W 
      := 
    V \times V' \times \mathbb{R}^1
  $$
  and
  $$
    \begin{tikzcd}[row sep=-2pt, column sep=0pt]
      V \ar[rr, "{ \alpha }"]
      && 
      W
      \\
      (t, \vec x\,) 
        &\longmapsto& 
      \left(
        (t, \vec x\,), 
        (t, 0),
        0
      \right)
      \mathrlap{,}
    \end{tikzcd}
    \;\;\;\;\;\;
    \begin{tikzcd}[row sep=-3pt, column sep=0pt]
      V' \ar[rr, "{ \alpha' }"]
      && 
      W
      \\
      (t',\vec x{\,}') 
        &\longmapsto& 
      \left(
        (t', 0),
        (t', \vec x{\,}'),
        t'^2
      \right)
      \mathrlap{.}
    \end{tikzcd}
  $$
  Since
  $$
    \begin{aligned}
      \alpha(0,0) 
        & = 
      \left(
        (0,0),
        (0,0),
        0
      \right)
      \mathrlap{,}
      \\
      \alpha'(0,0) 
        & = 
      \left(
        (0,0),
        (0,0),
        0
      \right)
      \mathrlap{,}
    \end{aligned}
    \;\;\;\;\;\;
    \begin{aligned}
      \tfrac{\partial}{\partial t}
      \alpha(0,0)
      & =
      \left(
        (1,0),
        (1,0),
        0
      \right)
      \mathrlap{,}
      \\
      \tfrac{\partial}{\partial t'}
      \alpha'(0,0)
      & =
      \left(
        (1,0),
        (1,0),
        0
      \right)
      \mathrlap{,}
    \end{aligned}
  $$
  this makes the left part of our diagram commute.

  Finally, to define $\phi$, first consider
  \begin{align} 
    \delta 
      : 
    \mathbb{R} & \longrightarrow \mathbf{R}^\infty
\\[-1pt]
    t & \longmapsto f'(t, 0) - f(t, 0)
    \,.
  \end{align} 
  By the commutativity of the outer diagram, we have $\delta(0) = 0$ and $\tfrac{d}{d t} \delta(0) = 0$. Hence by Hadamard's lemma there is a smooth function 
  \begin{equation}
    \label{HadamardRemainderOfDeltaForD1Case}
    \inlinetikzcd{
    \mu 
      : 
    \mathbb{R} \ar[r] \& \mathbf{R}^\infty
    },
    \;\;
    \mbox{
      such that $\delta(t) = t^2 \cdot \mu(t)$.
    }
  \end{equation}
  With that function in hand, finally define
  \begin{equation}
    \label{DefOfPhiForD1Case}
    \phi 
    :
    \left(
      (t,\vec x\,), 
      (t',\vec x{\,}'),
      j
    \right)
    \longmapsto
    f(t, \vec x\,)
    -
    f'(t, 0)
    +
    f'(t, \vec x{\,}')
    +
    j
    \cdot
    \mu(t)
    \mathrlap{\,.}
  \end{equation}
  We check that this makes the right part of our diagram commute, since:
  $$
    \begin{aligned}
      \phi \circ \alpha(t, \vec x\,)
      &=
      \phi\left(
        (t, \vec x\,), 
        (t, 0),
        0
      \right)
      \\
      & =
      f(t, \vec x\,) + 0 + 0
      \mathrlap{\,,}
    \end{aligned}
  $$
  and
  $$
    \begin{aligned}
      \phi \circ \alpha'(t', \vec x{\,}')
      &=
      \phi\left(
        (t', 0), 
        (t', \vec x{\,}'),
        t'^2
      \right)
      \\
      & =
      \grayunderbrace{
      f(t',0)
      -
      f'(t',0)
      }{\mathclap{%
       -\delta(t')
      }}
      +
      f'(t',\vec x{\,}')
      +
      \grayunderbrace{
      t'^2
      \cdot
      \mu(t')
      }{\mathclap{%
        \delta(t')
      }
      }
      \\
      & =
      f'(t',\vec x{\,}')
      \mathrlap{\,.}
    \end{aligned}
  $$
  This completes the proof.
\end{proof}

Next, we generalize this proof to the case where (still $U \defneq \ast$ but) $\mathbb{D}$ is arbitrary. To that end, choose an inclusion $\inlinetikzcd{\mathbb{D} \ar[r, hook] \& \mathbb{R}^d}$ with minimal $d$ and choose a set of generators
\begin{equation}
  \label{GeneratorsForKernel}
  \vec h
  :=
  \left(
    h_i 
    \in
    C^\infty(\mathbb{R}^d)
  \right)_{i = 1}^{n}
  \;\;
  \left(\mbox{for some}\;n \in \mathbb{N}\right)
\end{equation}
for the kernel of the dual algebra homomorphism
$
  \begin{tikzcd}[sep=15pt]
    C^\infty(\mathbb{R}^d)
    \ar[
      r,
      ->>
    ]
    &
    C^\infty(\mathbb{D})
    \mathrlap{\,.}
  \end{tikzcd}
$

\begin{lemma}
\label[lemma]
  {TowardsFactoredMapsToRInfinityEquivalences}
  Given a commuting diagram of smooth functions of the form of the outer solid arrows in the following diagram:
  $$
  \begin{tikzcd}[
    row sep=12pt,
    column sep=40
  ]
    &
    V
    \ar[
      dr,
      "{ f }"
    ]
    \ar[
      d,
      dashed,
      "{ \alpha }"
    ]
    \\
    \mathbb{D}
    \ar[
      ur,
      "{ \iota }"
    ]
    \ar[
      dr,
      "{ \iota' }"{swap}
    ]
    \ar[
      r,
      dashed
    ]
    &
    W
    \ar[
      r,
      dashed,
      "{ \phi }"{description, pos=.45}
    ]
    &
    \mathbf{R}^\infty,
    \\
    &
    V'
    \mathrlap{\,,}
    \ar[
      ur,
      "{ f' }"{swap}
    ]
    \ar[
      u,
      dashed,
      "{ \alpha' }"{swap}
    ]
  \end{tikzcd}
$$
where $\iota$ and $\iota'$ are monomorphisms,
then there exists $W \in \CartesianSpaces$ and dashed maps also making the inner diagram commute.
\end{lemma}
\begin{proof}
  Without restriction of generality we may assume that $\iota(0) = 0 \in V$ and $\iota'(0) = 0 \in V'$ (otherwise reparametrize). 
  Since $\iota$ and $\iota'$ are monomorphic, we may decompose $V$, $V'$ respectively into the span of the image of $\iota$, $\iota'$ and its orthogonal complement, respectively:
  $$
    \begin{tikzcd}[row sep=0pt]
      \mathbb{R}^d \times V_\perp
      \ar[r, "{ \sim }"]
      &
      V
      \mathrlap{\,,}
      \\
      \mathbb{R}^d \times V'_\perp
      \ar[r, "{ \sim }"]
      &
      V'
      \mathrlap{\,.}
    \end{tikzcd}
  $$
  Now take
  $$
    W 
    :=
    V 
    \times
    V'
    \times
    \mathbb{R}^n
  $$
  and
  $$
    \begin{aligned}
      \alpha(\vec t\,, \vec x\,)
      & :=
      \left(
        (\vec t\,, \vec x\,), 
        (\vec t\,, 0),
        0
      \right)
      \\
      \alpha'(\vec t{\phantom{.}}', \vec x{\,}')
      &
      :=
      \left(
        (\vec t{\phantom{.}}',0)
        ,\,
        (\vec t{\phantom{.}}', \vec x{\,}'),
        \vec h(\vec t{\phantom{.}}') 
      \right)
      \mathrlap{,}
    \end{aligned}
  $$
  where the last component are the values of the generators \cref{GeneratorsForKernel} of the nil ideal.
  
  Then we have the following composites (where in the lower lines we show the pullback of the canonical coordinate functions along the top map, whose inspection is sufficient for identifying these maps, cf. \cite[Prop. 4.6]{GS26-FieldsII})
  $$
    \begin{tikzcd}[
      column sep=-7pt,
      row sep=-5pt
    ]
      \mathbb{D}
      \ar[r, "{ \iota }"]
      &[25pt]
      \mathbb{R}^d 
        \times 
      V_\perp
      \ar[r, "{ \alpha }"]
      &[25pt]
      \mathbb{R}^d 
        &\times& 
      V_\perp 
        &\times&
      \mathbb{R}^d 
        &\times& 
      V'_\perp
      &\times&
      \mathbb{R}^n
      \\
      {[\,\vec t \;]}
      \ar[
        r,
        <-|,
        shorten=6pt
      ]
      &
      \vec t
      \ar[
        r,
        <-|,
        shorten=6pt
      ]
      &
      \vec t 
      \\
      0
      \ar[
        r,
        <-|,
        shorten=4pt
      ]
      &
      \vec x
      \ar[
        rrr,
        <-|,
        shorten=4pt
      ]
      & 
      &&
      \vec x
      \\
      {[\,{\vec t}{\phantom{.}}']}
      \ar[
        r,
        <-|,
        shorten=4pt
      ]
      &
      {\vec t}{\phantom{.}}'
      \ar[
        rrrrr,
        <-|,
        shorten=4pt
      ]
      & 
      &&
      &&
      {\vec t}{\phantom{.}}'
      \\
      0
      \ar[
        r,
        <-|,
        shorten=4pt
      ]
      &
      0
      \ar[
        rrrrrrr,
        <-|,
        shorten=4pt
      ]
      & 
      &&
      &&
      &&
      {\vec x}{\,}'
      \\
      0
      \ar[
        r,
        <-|,
        shorten=4pt
      ]
      &
      0
      \ar[
        rrrrrrrrr,
        <-|,
        shorten=4pt
      ]
      & 
      &&
      &&
      &&
      &&
      \vec j
    \end{tikzcd}
  $$
  and
  $$
    \begin{tikzcd}[
      column sep=-7pt,
      row sep=-5pt
    ]
      \mathbb{D}
      \ar[r, "{ \iota' }"]
      &[25pt]
      \mathbb{R}^d 
        \times 
      V'_\perp
      \ar[r, "{ \alpha' }"]
      &[25pt]
      \mathbb{R}^d 
        &\times& 
      V_\perp 
        &\times&
      \mathbb{R}^d 
        &\times& 
      V'_\perp
        &\times&
      \mathbb{R}^n
      \\
      {[\,\vec t\;]}
      \ar[
        r,
        <-|,
        shorten=6pt
      ]
      &
      \vec t
      \ar[
        r,
        <-|,
        shorten=6pt
      ]
      &
      \vec t 
      \\
      0
      \ar[
        r,
        <-|,
        shorten=4pt
      ]
      &
      0
      \ar[
        rrr,
        <-|,
        shorten=4pt
      ]
      & 
      &&
      \vec x
      \\
      {[\,{\vec t}{\phantom{.}}']}
      \ar[
        r,
        <-|,
        shorten=4pt
      ]
      &
      {\vec t}{\phantom{.}}'
      \ar[
        rrrrr,
        <-|,
        shorten=4pt
      ]
      & 
      &&
      &&
     {\vec t}{\phantom{.}}'
      \\
      0
      \ar[
        r,
        <-|,
        shorten=4pt
      ]
      &
      {\vec x}{\,}'
      \ar[
        rrrrrrr,
        <-|,
        shorten=4pt
      ]
      & 
      &&
      &&
      &&
{ \vec x}{\,}'
      \\
      0
      \ar[
        r,
        <-|,
        shorten=4pt
      ]
      &
      \vec h({\vec t}{\phantom{.}}')
      \ar[
        rrrrrrrrr,
        <-|,
        shorten=4pt
      ]
      & 
      &&
      &&
      &&
      &&
      \vec j
      \mathrlap{\,,}
    \end{tikzcd}
  $$
  where in the last line we use that the functions $\vec h$ \cref{GeneratorsForKernel} vanish on $\mathbb{D}$, by assumption.  
  Since these two composite maps hence agree, the left part of our diagram commutes.

  Now consider
  $$
    \begin{tikzcd}[row sep=-2pt, column sep=0pt]
      \mathbb{R}^d
      \ar[rr, "{ \delta }"]
      &&
      \mathbf{R}^\infty
      \\
      \vec t 
      &\longmapsto&
      f'(\vec t\,, 0)
      -
      f(\vec t\,, 0)\,.
    \end{tikzcd}
  $$
  By the commutativity of the outer part of our diagram,
  the components of this map are in the ideal spanned by the $h_i$ \cref{GeneratorsForKernel}, so that
  \begin{equation}
    \label{HadamardRemainderOfDelta}
    \delta(\vec t\,)
    =
    \textstyle{\sum_i}
    \,
    h_i(\vec t\,)
    \,
    \mu^i(\vec t\,)
    \;\;\;
    \mbox{for some}
    \;\;\;
    (\,
      \inlinetikzcd{
      \mu^i
      :
      \mathbb{R}^d
      \ar[r]
      \&
      \mathbf{R}^\infty
      }
    )_{i = 1}^{n}
    \mathrlap{\,.}
  \end{equation}
  With this in hand, finally define
  \begin{equation}
    \label{DefinitionOfPhi}
    \phi\left(
      (\vec t\,, \vec x \,)
      ,\,
      ({\vec t}{\phantom{.}}',\vec x{\,}'\,)
      ,\
      \vec j \,
    \right)
    :=
    f(\vec t\,, \vec x \,)
    -
    f'(\vec t\,, 0)
    +
    f'(\vec t\,,\vec x{\,}')
    +
    \textstyle{\sum_i}
    \,
    j_i
    \,
    \mu^i(\vec t\,)
    \mathrlap{\,.}
  \end{equation}
  We check by direct computation that this indeed makes the right part of our diagram commute:
  $$
    \begin{aligned}
      \phi \circ \alpha(
        \vec t\,, \vec x \,
      )
      &
      =
      \phi\left(
        (\vec t\,, \vec x \,)
        ,\,
        (\vec t\,, 0)
        ,\,
        0
      \right)
      \\
      & =
      f(\vec t\,, \vec x\,)
      +
      0
      +
      0
    \end{aligned}
  $$
  and
  $$
    \begin{aligned}
      \phi \circ \alpha'(
        {\vec t}{\phantom{.}}', \vec x{\,}'
      )
      &
      =
      \phi\left(
        ({\vec t}{\phantom{.}}', 0)
        ,\,
        ({\vec t}{\phantom{.}}', \vec x{\,}')
        ,\,
        \vec h(\vec t\,)
      \right)
      \\
      & =
      \grayunderbrace{
      f(\vec t\,, 0)
      -
      f'(\vec t\,, 0)
      }{\mathclap{%
        -\delta(\vec t\,)
      }}
      +
      f'(\vec t\,, \vec x{\,}')
      +
      \grayunderbrace{%
        \textstyle{\sum_i}
        \,
        h_i(\vec t\,)
        \,
        \mu^i(\vec t\,)
      }{{\mathclap{
        \delta(\vec t\,)
      }}}
      \\
      & =
      f'(\vec t\,, \vec x{\,}')
      \mathrlap{\,.}
    \end{aligned}
  $$
  This completes the proof.
\end{proof}

Finally, it is not much more work to lift this proof to full generality:

\begin{proof}[Proof of \cref{FactoredMapsToRInfinityEquivalences}]
By \cref{RectifyingProperEmbeddingsOfCartesianSpaces} we may assume, without restriction of generality, that, up to diffeomorphism on the target: 
$$
  \begin{array}{l}
    \inlinetikzcd{
      \iota \vert_{U\times \{*\}} \simeq (\mathrm{id}_,0) : U \ar[r, hook] \& U \times \tilde V \simeq V}
    \\[-2pt]
    \inlinetikzcd{
      \iota' \vert_{U\times \{*\}} \simeq (\mathrm{id},0) : U \ar[r, hook] \& U \times \tilde V' \simeq V'}
      \mathrlap{.}
  \end{array}
$$
But under this identification, we have over each point $u \in U$ the exact situation of \cref{TowardsFactoredMapsToRInfinityEquivalences}, and hence the solution is as given there, just with $W$ having another factor of $U$ and with $\alpha$, $\alpha'$ and $\phi$ depending on a further variable $u \in U$.

Concretely, in generalizing \cref{FirstStepTowardsFactoredMapsToRInfinityEquivalences} this way: 
\begin{enumerate}
\item The parametrization of $V$ is now as a bundle over $U\times \mathbb{R}$, with $(u, t\cdot v(u), \vec x_{v(u)}) \in V$ , where $\vec x_{v(u)}$ are  elements perpendicular to $v(u)$.

\item The corresponding $\delta$ is now a map $U\times \mathbb{R} \rightarrow \mathbf{R}^\infty$ which satisfies $\delta(u,0) = 0$, $\frac{\partial}{\partial t} \delta (u,0) =0$ and so, by the \emph{partial/parameterized version} of Hadamard's lemma \cref{HadamardsLemma}, may be written as $\delta(u,t) = t^2 \cdot \mu(u,t)$ for a smooth $\mu : U\times \mathbb{R} \rightarrow \mathbf{R}^\infty$. \qedhere
\end{enumerate}
\end{proof}

\appendix
%%%%%%%%%%%%%%%%%%%%%%%%%%%%
\section{Background}
%%%%%%%%%%%%%%%%%%%%%%%%%%%%

For reference in the main text, here we briefly record some basic definitions and facts.

%%%%%%%%%%%%%%%%%%%%%%%%%%%%%%%
\subsection{Formal Smooth Sets}
\label{FormalSmoothSets}
%%%%%%%%%%%%%%%%%%%%%%%%%%%%%%%

We briefly review the
topos of \emph{formal smooth sets} (\parencites[Def. 4.5.5]{Sc13-dcct}[\S 2.1]{KS26}[\S 2]{GS26-FieldsII}, exposition in \cite{Schreiber2025}), equivalent (by \cite[\S 4.1]{GS26-FieldsII}) to Dubuc's \emph{Cahiers topos} \parencites{Dubuc1979}{KockReyes1987}, which is a ``well-adapted model'' for synthetic differential geometry. 

\smallskip 
Our ground field is the real numbers. We write $\mathrm{CAlg}$ for the category of commutative $\mathbb{R}$-algebras.

\begin{definition}
\label[definition]{CategoryOfCartesianSpaces}
\begin{itemize} 
\item[\bf (i)] We write $\CartesianSpaces$ for the category whose objects are the Cartesian spaces $\mathbb{R}^n$, $n \in \mathbb{N}$, and whose morphisms are the smooth functions between these. Hence this is the full subcategory of smooth manifolds
  $$
    \begin{tikzcd}
      \CartesianSpaces
      \ar[r, hook]
      &
      \mathrm{SmthMfd}
    \end{tikzcd}
  $$
  on the Cartesian spaces. We will denote generic Cartesian spaces by $U, V \in \CartesianSpaces$. 

 \item[\bf (ii)]  We regard this category as a site by equipping it with the coverage (Grothendieck pre-topology) of differentiably good open covers.
 \end{itemize} 
\end{definition}
\begin{proposition}[{cf. \cite[\S 35]{KolarMichorSlovak1993}}]
\label[proposition]{MilnorExercise}
  The functor which forms (plain) $\mathbb{R}$-algebras of smooth functions constitutes a full inclusion of smooth manifolds, hence in particular of Cartesian spaces \textup{(\cref{CategoryOfCartesianSpaces})}:
  $$
    \begin{tikzcd}
      \CartesianSpaces
      \ar[r, hook]
      &
      \mathrm{SmthMfd}
      \ar[
        rr, 
        hook, 
        "{ C^\infty(-) }"
      ]
      &&
      \mathrm{CAlg}^{\mathrm{op}}
    \end{tikzcd}
  $$
\end{proposition}

\begin{definition}
\label[definition]{HaloedPoints}
The category of \emph{infinitesimally thickened points} (\emph{haloed points}) is the full subcategory of the opposite of commutative algebras
$$
  \begin{tikzcd}[sep=0pt]
    \mathrm{FrmPnt}
    \ar[rr, hook]
    &&
    \mathrm{CAlg}^{\mathrm{op}}
    \\
    \mathbb{D} 
    &\longmapsto&
    C^\infty(\mathbb{D})
    \mathrlap{\,.}
  \end{tikzcd}
$$
on those what are called \emph{Weil algebras} in the SDG literature or \emph{Artin local algebras} in algebraic geometry, namely those of the form
$$
  C^\infty(\mathbb{D})
  \simeq_{\mathbb{R}}
  \mathbb{R} \oplus V
  \mathrlap{\,,}
$$
where $V \subset C^\infty(\mathbb{D})$ is a finite-dimensional ideal which is \emph{nilpotent} in that
$$
  \exists_{k \in \mathbb{N}}
  \;:\;
  V^{k+1} = 0
  \mathrlap{\,.}
$$
The smallest $k$ for which this holds is the \emph{order of nilpotency} of $\mathbb{D}$.
\end{definition}
\begin{example}
\label[example]{InfinitesimalDisks}
  An \emph{infinitesimal disk} of dimension $d \in \mathbb{N}$ and order of infinitesimality $k \in \mathbb{N}$ is an infinitesimally thickened point
  $
    \mathbb{D}^d(k)
    \in
    \mathrm{FrmPnt}
  $ (\cref{HaloedPoints}) of the form
  $$
    \begin{aligned}
    C^\infty\big(\mathbb{D}^d(k)\big)
    & \simeq
    \mathbb{C}^{\infty}(\mathbb{R}^d)
    /
    (x^1, \cdots, x^d)^{k+1}
    \\
    &
    \simeq
    \mathbb{R}[\epsilon^1, \cdots, \epsilon^d]
    /
    (\epsilon^1, \cdots, \epsilon^d)^{k+1}
    \mathrlap{\,.}
    \end{aligned}
  $$
\end{example}
The second equivalence above follows by Hadamard's lemma, which we repeatedly refer to. For reference, we state it in a version that covers both parameters and partial derivatives:
\begin{lemma}[Hadamard's Lemma, cf. {\cite[Lem. 2.2.7]{duistermaat-kolk-1}, we follow \cite[Lem. 4.1]{GS26-FieldsII}}]\label{HadamardsLemma}
For any smooth function $f(x,y)\in C^\infty(\mathbb{R}^k\times\mathbb{R}^m)$ and any $l\in \mathbb{N}$, there exist smooth functions $\{h_{i_{1}\cdots i_{l+1}}\}_{i_1,\cdots, i_{l+1}=1,\cdots, m}\subset C^\infty(\mathbb{R}^k\times \mathbb{R}^m)$ such that 
\begin{align*}
f&=  f(x,0) + \sum_{i=1}^m y^i \partial_{y^i} f(x,0) + \tfrac{1}{2} \sum_{i,j=1}^m y^i y^j \partial_{y^i}\partial_{y^j} f(x,0)
\\
 & \qquad {}
 + \cdots
 + \; \tfrac{1}{l!} \sum_{i_{1},\cdots,i_{l}=1}^m y^{i_{1}}\cdots y^{i_{l}}\cdot \partial_{y^{i_{1}}}\cdots \partial_{y^{i_{l}}}f(x,0)
\\
 & \quad {}
 + \sum_{i_{1},\cdots,i_{l+1}=1}^m y^{i_{1}}\cdots y^{i_{l+1}}\cdot h_{i_{1}\cdots i_{l+1}} (x,y) .
\end{align*}
\end{lemma}

\begin{definition}
\label[definition]{FormalCartesianSpaces}
 
  \begin{itemize} 
  \item[\bf (i)] The category of \emph{formal Cartesian spaces} is the full subcategory of the opposite of commutative algebras,
  $$
    \begin{tikzcd}[sep=0pt]
      \HaloedCartesianSpaces
      \ar[
        rr, 
        hook,
        "{
          C^\infty(-)
        }"
      ]
      &&
      \mathrm{CAlg}^{\mathrm{op}}
      \\
      U \times \mathbb{D}
      &\longmapsto&
      C^\infty(U)
      \otimes_{\mathbb{R}}
      C^\infty(\mathbb{D})
      \mathrlap{\,,}
    \end{tikzcd}
  $$
  on those which are tensor products of an algebra of smooth functions on a $U \in \CartesianSpaces$ (cf. \cref{MilnorExercise}) with those on a $\mathbb{D} \in \mathrm{FrmPnt}$ (\cref{HaloedPoints}).

\item[\bf (ii)]  We regard this category as a site by equipping it with the coverage (Grothendieck pre-topology) whose covering families are of the form
  $$
    \big\{
    \begin{tikzcd}
      U_i \times \mathbb{D}
      \ar[
        rr,
        hook,
        "{ \iota_i \times \mathrm{id} }"
      ]
      &&
      U \times \mathbb{D}
    \end{tikzcd}
    \big\}_{i \in I}
  $$
  for
  $$
    \{\,
    \begin{tikzcd}
      U_i
      \ar[
        r,
        hook,
        "{ \iota_i }"
      ]
      &
      U
    \end{tikzcd}
    \, \}_{i \in I}
  $$
  a covering family in $\CartesianSpaces$ (a differentiably good open cover).
  \end{itemize}
\end{definition}

By construction, we have a fully faithful inclusion of sites:
\begin{equation}
  \label{CrtSpInsideFrmCrtSp}
  \begin{tikzcd}[row sep=-2pt, column sep=0pt]
    \CartesianSpaces
    \ar[
      rr,
      hook,
      "{ i }"
    ]
    &&
    \HaloedCartesianSpaces
    \\
    U 
      &\longmapsto&
    U \times \mathbb{D}^0
    \mathrlap{\,.}
  \end{tikzcd}
\end{equation}

\begin{definition}
\label[definition]
  {ToposOfFormalSmoothSets}
We say that the sheaf topos 
\begin{enumerate}
\item over $\CartesianSpaces$ is the category of \emph{smooth sets},
\item over $\HaloedCartesianSpaces$ is the category of \emph{formal smooth sets},
\end{enumerate}
and that the essential image of the former under the left base change $i_!$ along the inclusion \cref{CrtSpInsideFrmCrtSp} constitute the \emph{reduced} formal smooth sets:
\begin{equation}
  \label{SmthSetInsideFrmSmthSet}
  \begin{tikzcd}[
    sep=5pt
  ]
    \SmoothSets
    :=
    \mathrm{Sh}(\CartesianSpaces)
    \ar[
      r,
      shift left=6pt,
      hook,
      "{ 
        i_! 
      }"{description}
    ]
    \ar[
      r,
      phantom,
      "{ \bot }"{scale=.6}
    ]
    \ar[
      r,
      shift right=7pt,
      <-,
      "{ i^\ast }"{description}
    ]
    &[30pt]
    \mathrm{Sh}(\HaloedCartesianSpaces)
    =:
    \HaloedSmoothSets
    \mathrlap{\,,}
  \end{tikzcd}
\end{equation}
\end{definition}
Here $i_!$ on sheaves is the \emph{left Kan extension} $i_!$ of presheaves \cref{CoendFormulaFori!} followed by sheafification $L$, and we are using that $i^\ast$ manifestly preserves sheaves:
\begin{equation}
  \label{i!AndSheafification}
  \begin{tikzcd}
    \mathrm{PSh}(\mathrm{CrtSp})
    \ar[
      rr,
      shift left=7pt,
      "{ i_! }"{description}
    ]
    \ar[
      rr,
      phantom,
      "{ \bot }"{scale=.6}
    ]
    \ar[
      rr,
      <-,
      shift right=7pt,
      "{ i^\ast }"{description}
    ]
    &&
    \mathrm{PSh}(\mathrm{FrmCrtSp})
    \\
    \mathrm{Sh}(\mathrm{CrtSp})
    \ar[
      u, 
      hook'
    ]
    \ar[
      rr,
      shift left=8pt,
      "{ i_! }"{description}
    ]
    \ar[
      rr,
      phantom,
      "{ \bot }"{scale=.6}
    ]
    \ar[
      rr,
      <-,
      shift right=7pt,
      "{ i^\ast }"{description}
    ]
    &&
    \mathrm{Sh}(\mathrm{FrmCrtSp})
    \mathrlap{\,.}
    \ar[
      u,
      hook',
      shift left=6pt
    ]
    \ar[
      u,
      phantom,
      "{ \bot }"{scale=.6, rotate=-90}
    ]
    \ar[
      u,
      <-,
      shift right=6pt,
      "{ L }"{swap}
    ]
  \end{tikzcd}
\end{equation}

Now, basic category theory shows the following important properties:
\begin{lemma}
\label[lemma]
  {PropertiesOfSmthSetInsideFrmSmthSet}
The inclusion $i_!$ in \cref{SmthSetInsideFrmSmthSet}:
\begin{enumerate}
\item
preserves representables $\inlinetikzcd{V \in \mathrm{CrtSp} \ar[r, hook] \& \mathrm{SmthSet}}$ in that:%
\footnote{We are leaving the Yoneda embedding notationally implicit.}
\begin{equation}
  \label{i!PreservesRepresentables}
  \inlinetikzcd{
  i_! V \simeq V \defneq V \times \mathbb{D}^ 0
  \in
  \HaloedCartesianSpaces
  \ar[r, hook]
  \&
  \HaloedSmoothSets
  \mathrlap{\,,}
  }
\end{equation}

\item
is fully faithful, hence 
\begin{equation}
  \label{iAstIExc}
  i^\ast \circ i_!
  \simeq
  \mathrm{id}
  \mathrlap{\,,}
\end{equation}

\item
and such that 
the \emph{reduced} objects $\inlinetikzcd{X \in \SmoothSets \ar[r, hook, "{ i_! }"] \& \HaloedSmoothSets}$ are equivalently those which are isomorphic to their \emph{purely reduced aspect} $\reduced(X)$,
\begin{equation}
  \label{BeingReduced}
  X \simeq \reduced(X)
  \mathrlap{\,,}
\end{equation}
where
\begin{equation}
  \label{IExciAst}
  \reduced
  :=
  i_! \circ i^\ast
  \mathrlap{\,.}
\end{equation}
\end{enumerate}
\end{lemma}
\begin{proof}
  The first statement follows with the coend formula \cref{CoendFormulaFori!}: Every pair $(\iota,f)$ (of maps into a Cartesian space $V$ factoring through a Cartesian space $W$) is equivalent to its actual composite regarded as the pair $(f \circ \iota, \mathrm{id})$, like this:
  $$
    \begin{tikzcd}[
      row sep=-2pt, column sep=30pt
    ]
      & 
      W
      \ar[
        dr,
        "{ f }"
      ]
      \ar[
        dd,
        "{ f }"
      ]
      \\
      U \times \mathbb{D}
      \ar[
        ur,
        "{ \iota }"
      ]
      \ar[
        dr,
        "{
          f \circ \iota
        }"{sloped, swap}
      ]
      &&
      V .
      \\
      & 
      V
      \ar[
        ur, 
        "{ \mathrm{id} }"{swap}
      ]
    \end{tikzcd}
  $$
  This shows that $i_!$ on presheaves preserves representables. Finally, our coverage is clearly subcanonical in that representable presheaves are already sheaves.

  For the second statement cf.  \cite[Prop. 4.23]{Kelly1982} with \cite[\S IV.3 Thm. 1]{MacLane1998}. From this the last statement follows.
\end{proof}

A classical notion in infinite-dimensional differential geometry are \emph{Fr{\'e}chet manifolds} (cf. \cite[\S I.4]{Hamilton1982}). These are faithfully subsumed here:
\begin{lemma}
\label[lemma]{FrechetManifoldsAsSmoothSets}
  The evident inclusion of Fr{\'e}chet manifolds among smooth sets \textup(and hence further among formal smooth sets \cref{ToposOfFormalSmoothSets}) 
  \begin{equation}
    \label{FrechetManifoldsFullyFaithfulInSmoothSets}
    \begin{tikzcd}[sep=0pt]
      \mathrm{FrchtMfd}
      \ar[rr, hook]
      &&
      \mathrm{SmthSet}
      \\
      X 
        &\longmapsto&
      \left(
        U 
          \mapsto
        C^\infty(U,X)
      \right)
    \end{tikzcd}
  \end{equation}
  is fully faithful.
\end{lemma}
\begin{proof}
  This is essentially the statement of \cite[Thm. 3.1.1]{Losik1994}, cf. \cite[Prop. 2.22]{KS26}.
\end{proof}

%%%%%%%%%%%%%%%%%%%%%%%%%%%%%%%%%%%%%
\subsection{Formal embeddings}
\label{OnFormalProperEmbeddings}
%%%%%%%%%%%%%%%%%%%%%%%%%%%%%%%%%%%%%

\begin{definition} 
  \label[definition]{FormalEmbedding}
  A map $\inlinetikzcd{i : U \times \mathbb{D} \ar[r] \& V}$ in $\HaloedCartesianSpaces$ (\cref{FormalCartesianSpaces}) is a \emph{formal embedding} if its restriction to $U \times \{0\}$ is an embedding of smooth manifolds and its restriction to each $\{u\} \times \mathbb{D}$ is a monomorphism. Moreover, this is a \emph{proper} formal embedding if the restriction to $U \times \{0\}$ is also a proper map.
\end{definition}
The following lemma is immediate but still worth recording:
\begin{lemma}
  \label[lemma]{FirstMapMayBeAssumedFormalEmbedding}
  Every $\inlinetikzcd{U \times \mathbb{D} \ar[r, "i"] \& V \ar[r, "{f}"] \& X }$ is equivalent \cref{GeneratingRelation} to $(\iota,f')$, where $\iota$ is a formal proper embedding \textup{(\cref{FormalEmbedding})} of arbitrarily high codimension.
\end{lemma}
\begin{proof}
 Choose a monomorphism (where $d$ may be arbitrarily large) of the form
 $\inlinetikzcd{\delta : U \times \mathbb{D} \ar[r, hook] \& U \times \mathbb{R}^d}$ (cf. \cite[Cor. 4.3(iii)]{GS26-FieldsII}), note that this makes the following diagram commute:
 $$
  \begin{tikzcd}[
    row sep=1pt,
    column sep=5em
  ]
    & V 
    \ar[dr, "{ f }"]
    \\
    U \times \mathbb{D}
    \ar[ur, "{ i }"]
    \ar[
      dr, 
      "{ 
        \iota 
          := 
        (i, \delta) 
      }"{sloped, swap}
    ]
    && 
    X
    \mathrlap{\,,}
    \\
    & V \times (U \times \mathbb{R}^d)
    \ar[
      ur,
      "{
         f' := f \circ \mathrm{pr}_V
      }"{swap, sloped}
    ]
    \ar[
      uu,
      "{  
        \mathrm{pr}_V
      }"{pos=.4}
    ]
  \end{tikzcd}
 $$
 and finally observe that $\iota := (i, \delta)$ is a formal proper embedding.
\end{proof}

\begin{lemma}
\label[lemma]
  {RectifyingProperEmbeddingsOfCartesianSpaces}
  For every proper embedding $\inlinetikzcd{\iota : \mathbb{R}^n \ar[r, hook] \& \mathbb{R}^{n+ c}}$ of sufficiently high codimension $c$, there exists a diffeomorphism of $\mathbb{R}^{n + c}$ under which the embedding is a rectilinear inclusion: $i \simeq (\mathrm{id}, 0)$.
\end{lemma}
\begin{proof}
  First we claim that there exists a proper smooth homotopy from $i$ to $(\mathrm{id}, 0)$. This may be seen by observing that proper maps extend to pointed maps between one-point compactifications (cf. \cite[p. 80]{James1984}). In our case the extended maps are of the form $\inlinetikzcd{S^n \ar[r] \& S^{n + c}}$ and hence pointed null-homotopic for $c \geq 1$. This pointed null-homotopy comes from a proper homotopy between the original maps as desired.
  
  From this it follows that $i$ is related to $(\mathrm{id}, 0)$ by ambient isotopy (with \cite[\S 8.1 Exc. 11]{Hirsch1976}, which requires $c > 1 + n$) and thus by diffeotopy (with \cite[\S 8.1 Thm. 1.3]{Hirsch1976}).
\end{proof}

%%%%%%%%%%%%%%%%%%%%%%%%%%%%%%
\subsection{Jet Bundles}
\label{OnJetBundles}
%%%%%%%%%%%%%%%%%%%%%%%%%%%%%%

For traditional background on jet bundle geometry cf. \parencites{Saunders1989}{Takens1979}.

A finite-rank \emph{jet bundle} $J^k_\Sigma E$ is the bundle of all possible partial derivatives, of order $\leq k$, of sections of a given surjective submersion $\inlinetikzcd{p : E \ar[r] \& \Sigma}$ of smooth manifolds. 

This traditional notion has a slick synthetic differential formulation (\cite[\S 3.3]{KS26}\cite[\S 2.4]{GS26-FieldsII}) in terms of  the maximal infinitesimal disks $\inlinetikzcd{\mathbb{D}_s(k) \ar[r] \& \Sigma }$ (\cref{InfinitesimalDisks}) around any point $s \in \Sigma$: The fiber of $J^k_\Sigma E$ over $s$ consists of the sections of $p$ over $\mathbb{D}_s(k)$:
$$
  (J^k_\Sigma E)_s
  \simeq
  \mathrm{FrmSmthSet}_{/\Sigma}\left(
    \mathbb{D}_s(k)
    ,\,
    E
  \right)
  \defneq
  \left\{
  \begin{tikzcd}
    & 
    E
    \ar[d, ->>, "{ p }"{pos=.4}]
    \\
    \mathbb{D}_s(k)
    \ar[r, hook]
    \ar[
      ur, 
      dashed
    ]
    &
    \Sigma
  \end{tikzcd}
  \, \right\}
  \mathrlap{\,,}
$$
and $J^k_\Sigma E$ canonically inherits the structure of a smooth manifold this way.

Now, the evident inclusions $\inlinetikzcd{\mathbb{D}_s(k) \ar[r, hook] \& \mathbb{D}_s(k+1)}$ thus induce natural projections $\inlinetikzcd{J^{k+1}_\Sigma E \ar[r, ->>] \& J^k_\Sigma E}$, and the \emph{infinite rank} jet bundle (cf. \cite[\S 7]{Saunders1989}) is the projective limit over these projections:
\begin{equation}
  \label{InfiniteJetBundleInBackground}
  J^\infty_\Sigma E
  :=
  \varprojlim_{k \in \mathbb{N}}
  J^k_\Sigma E
  \mathrlap{\,.}
\end{equation}
Traditionally, this limit is understood in the category of Fr{\'e}chet manifolds (cf. \cref{FrechetManifoldsAsSmoothSets}), and it is not \emph{a priori} clear how this relates to the corresponding limit formed in formal smooth sets, since the inclusion $i_!$ \cref{SmthSetInsideFrmSmthSet} does not preserve limits in general 

That $i_!$ does preserve the particular limits \cref{InfiniteJetBundleInBackground} is the content of our main result \cref{i!PreservesInfiniteJetBundles}.

%%%%%%%%%%%%%%%%%%%%%%%
% end matter 
%%%%%%%%%%%%%%%%%%%%%%%

%%%%%%%%%%%%%%%%%%%%%%%%
\printbibliography
%%%%%%%%%%%%%%%%%%%%%%%%

\end{document}